\documentclass[pdflatex,sn-mathphys-ay]{sn-jnl}

\usepackage{amsmath,amssymb,amsthm}
\usepackage{graphicx}
\usepackage{bm}
\usepackage[ruled,vlined]{algorithm2e}
\usepackage{booktabs}

\newcommand{\T}[1]{\mathcal{#1}}    
\newcommand{\V}[1]{\bm{#1}}     
\newcommand{\M}[1]{\bm{#1}}     
\newcommand{\C}{\mathbb{C}}         
\newcommand{\R}{\mathbb{R}}         
\newcommand{\tr}{\operatorname{tr}}
\newcommand{\norm}[1]{\left\lVert#1\right\rVert_{F}} 
\newcommand{\bcirc}{\operatorname{bcirc}}

\newcommand{\diag}{\operatorname{diag}}

\theoremstyle{thmstyleone}
\newtheorem{theorem}{Theorem}[section]
\newtheorem{lemma}[theorem]{Lemma}
\newtheorem{proposition}[theorem]{Proposition}
\newtheorem{corollary}[theorem]{Corollary}

\theoremstyle{thmstyletwo}
\newtheorem{example}[theorem]{Example}

\theoremstyle{thmstylethree}
\newtheorem{definition}[theorem]{Definition}
\newtheorem{remark}[theorem]{Remark}

\numberwithin{equation}{section}
\begin{document}

\title[Hamiltonian and Symplectic Tensors in the T-product Algebra]{Hamiltonian and Symplectic Tensors in the T-product Algebra}

\author[1,2,3]{\fnm{Susana} \sur{L\'opez-Moreno}}
\author*[4]{\fnm{Taehyeong} \sur{Kim}}\email{thkim0519@knu.ac.kr}

\affil[1]{\orgdiv{Department of Mathematics}, \orgname{Pusan National University}, \orgaddress{\city{Busan}, \postcode{46241}, \country{Republic of Korea}}}
\affil[2]{\orgdiv{Humanoid Olfactory Display Center}, \orgname{Pusan National University}, \orgaddress{\city{Busan}, \postcode{46241}, \country{Republic of Korea}}}
\affil[3]{\orgdiv{Industrial Mathematics Center}, \orgname{Pusan National University}, \orgaddress{\city{Busan}, \postcode{46241}, \country{Republic of Korea}}}
\affil*[4]{\orgdiv{Nonlinear Dynamics and Mathematical Application Center}, \orgname{Kyungpook National University}, \orgaddress{\city{Daegu}, \postcode{41566}, \country{Republic of Korea}}}

\abstract{
We study Hamiltonian and symplectic tensor structures in the T-product algebra. 
We define T-Hamiltonian and T-symplectic tensors and characterize them through their Fourier-domain slices. 
For T-Hamiltonian tensors we establish the standard block form and the spectral symmetry of T-eigenvalues, while for T-symplectic tensors we derive the inverse and exponential-map properties. 
Our main result is a constructive T-Williamson normal form for tensors whose Fourier-domain slices are real symmetric positive-definite matrices. 
We also show that, under the Hermitian symplectic convention adopted here, this decomposition does not extend directly to arbitrary Hermitian positive-definite Fourier-domain slices, and we derive a real-valued recovery criterion under Fourier conjugate symmetry. 
Numerical experiments verify the construction, exhibit runtime trends consistent with the slice-wise complexity $O(pn^3)$, and illustrate the framework on a Fourier-domain encoding of covariance-matrix families arising in continuous-variable quantum dynamics.}



\maketitle

\section{Introduction}
\label{sec:intro}

Hamiltonian and symplectic structures play a central role in classical mechanics, control theory, and structure-preserving numerical analysis; see, for example, the classical references \cite{mehrmann1991autonomous,hairer2006geometric,arnol2013mathematical} and recent developments in \cite{sun2020structure,pagliantini2021dynamical,mehrmann2023control}. 
For matrix-valued problems, these structures encode spectral symmetries, invariants, and canonical forms that are essential both in theory and in computation (\cite{paige1981schur,mehrmann1991autonomous}). 

In many applications, however, one encounters not a single structured matrix but a family of matrices indexed by time, frequency, or a control parameter (\cite{weedbrook2012gaussian,adesso2014continuous}). 
Such problems call for an algebraic setting that stores the family as one object while preserving the structural features of the underlying matrix theory (\cite{bentbib2025tensor,be2025numerical}).

The CANDECOMP/PARAFAC and Tucker decompositions, introduced by \cite{kolda2009tensor, de2000multilinear}, are effective for compression and representation, but they lack a closed algebraic structure that directly extends matrix multiplication.
Consequently, they are not suitable for transferring Hamiltonian or symplectic structures to tensor-valued data. 
The T-product, given by \cite{kilmer2011factorization, kilmer2013third}, is a better framework for this purpose, as it provides a matrix-mimetic algebraic structure. 
Its Fourier-domain representation reduces tensor multiplication to independent matrix multiplications on frontal slices, which allows matrix concepts to be lifted to third-order tensors in a natural and computationally tractable way, as seen in \cite{kilmer2013third, ju2025geometric}.

Recent work in the T-product setting has developed the notions of T-Hermitian, T-unitary, and T-positive-definite tensors, together with geometric and computational tools, such as the T-geometric mean (\cite{ju2025geometric}). 
Related results for Hamiltonian and symplectic tensors have also been studied in an Einstein-product framework in connection with multilinear control systems (\cite{wang2024algebraic}). 
Nevertheless, a systematic Hamiltonian-symplectic theory in the T-product algebra, including a tensor analogue of Williamson's normal form, is still missing.

We formulate T-Hamiltonian and T-symplectic tensors in the T-product algebra and study their structural properties through the Fourier-domain slice representation. 
Our main result is a constructive T-Williamson normal form for tensors whose Fourier-domain slices are real symmetric positive-definite matrices. 
The construction is done slice-wise in the Fourier domain and yields an algorithm with arithmetic complexity of order $O(pn^3)$.
We also show that this result does not extend directly
to arbitrary Hermitian positive-definite Fourier-domain slices under the convention $S^HJS = J$.

The main contributions of this paper are summarized as follows:
\begin{itemize}
    \item We introduce T-Hamiltonian and T-symplectic tensors in the T-product algebra and characterize them by their Fourier-domain slices.
    \item We establish basic structural properties, including the block form and spectral symmetry for T-Hamiltonian tensors, and the inverse and exponential-map properties for T-symplectic tensors.
    \item We introduce a constructive T-Williamson normal form for tensors whose slices in the Fourier domain  are real symmetric positive-definite matrices, and provide an explicit algorithm for its computation. We also derive the associated slice-wise algorithm, and show that an arbitrary Hermitian positive-definite extension fails under the convention $S^H JS = J$.
    \item We derive a real-valued recovery criterion under Fourier conjugate symmetry and provide numerical validation, including an  example from continuous variable quantum dynamics.
\end{itemize}

This paper is organized as follows. 
Section~\ref{sec:prelim} reviews the necessary background on the T-product. 
Section~\ref{sec:framework} introduces our main theoretical framework, defining T-Hamiltonian and T-symplectic tensors and developing the T-Williamson normal form. 
Section~\ref{sec:application} contains numerical validation and a Fourier-domain quantum illustration of the proposed framework. 
Section~\ref{sec:conclusion} contains concluding remarks and several directions for future work.

\section{Preliminaries}
\label{sec:prelim}

In this paper, tensors are denoted by calligraphic letters (e.g., $\T{A}$), matrices by bold capital letters (e.g., $\M{A}$), and vectors by bold lowercase letters (e.g., $\V{a}$). 
The space of third-order tensors of size $m \times n \times p$ with complex entries is denoted by $\C^{m\times n\times p}$. The frontal slices of a tensor $\T{A} \in \C^{m \times n \times p}$ are denoted $\M{A}^{(k)} \in \C^{m \times n}$ for $k=1, \dots, p$.
The inner product between two tensors $\T{A} = [a_{ijk}]$ and $\T{B} = [b_{ijk}]$ in $\C^{m\times n\times p}$ is defined as $\langle \T{A}, \T{B} \rangle_{F} = \sum_{i,j,k} \bar{a}_{ijk}b_{ijk}$, where $\bar{a}_{ijk}$ is the complex conjugate of $a_{ijk}$. 
The Frobenius norm associated with this inner product is expressed as $\norm{\T{A}} = \sqrt{\langle \T{A},\T{A} \rangle_{F}}$.

The T-product algebra is based on the mapping of tensors to block circulant matrices (\cite{kilmer2013third}), and the \emph{T-product} of two third-order tensors $\T{A}*\T{B}$ can be defined as the operation such that $\bcirc(\T{A}*\T{B})=\bcirc(\T{A})\bcirc(\T{B})$, i.e., as the matrix multiplication of the corresponding block circulant matrices of the tensors.
The block-circulant representation is specially useful because it has a Fourier
diagonalization. After applying the DFT in the third mode, a block-circulant matrix
decomposes into independent matrix blocks, so tensor multiplication is reduced to slice-wise matrix multiplication.
This property is fundamental for the efficient computation of the T-product. 
In fact, the DFT is just one of several linear transformations that can be used to define a tensor product;
alternative tensor multiplications can also be defined using other linear maps, as shown by~\cite{kernfeld2015tensor,keegan2025projected,ju2025computation}.
In this paper, we will exclusively use the definition of the T-product as the result of applying the DFT transform. 
The following standard lemma from \cite{kilmer2011factorization} records the Fourier block-diagonalization of block circulant matrices.

Throughout this paper, we will use $\widehat{\T{A}}$ to denote the tensor $\T{A}$ in the Fourier domain.

\begin{lemma}
\label{LemmaFourierBlockDiag}
Let $\T{A} \in \C^{m \times n \times p}$. 
The block circulant matrix $\bcirc(\T{A})$ can be diagonalized as
\begin{equation}\label{eqFourierBlockDiagMN}
    \bcirc(\T{A})=(\M{F}_{p}^H \otimes \M{I}_{m}) \diag(\widehat{\M{A}}^{(1)},\widehat{\M{A}}^{(2)},\ldots,\widehat{\M{A}}^{(p)}) (\M{F}_{p} \otimes \M{I}_{n}),
\end{equation}
where $\otimes$ denotes the Kronecker product, $\M{I}_{k}$ is the $k\times k$ identity matrix, and $\M{F}_{p}$ is the $p\times p$ normalized DFT matrix with entries $(\M{F}_{p})_{ij} = \frac{1}{\sqrt{p}}\omega^{(i-1)(j-1)}$ for $\omega=e^{2 \pi \sqrt{-1}/p}$. Here, 
$\widehat{\M{A}}^{(i)}$ denotes the $i$-th frontal slice of $\widehat{\T{A}}$. Following the standard convention in the T-product literature (\cite{kilmer2011factorization,kilmer2013third}), the unitary similarity by $\M{F}_{p}$ produces diagonal blocks equal to the ordinary FFT coefficients of the frontal slices. The diagonal blocks $\widehat{\M{A}}^{(i)} \in \C^{m \times n}$ are therefore the matrices
\begin{equation*}
\widehat{\M{A}}^{(i)} = \sum_{j=1}^{p} \omega^{(i-1)(j-1)}\M{A}^{(j)}, \quad i=1,\dots,p,
\end{equation*}
with inverse transformation
\begin{equation} \label{eq:ifft}
\M{A}^{(j)} = \frac{1}{p}\sum_{i=1}^{p} \omega^{-(i-1)(j-1)}\widehat{\M{A}}^{(i)}, \quad j=1,\dots,p.
\end{equation}
\end{lemma}

Lemma~\ref{LemmaFourierBlockDiag} provides an equivalent, computationally efficient interpretation of the T-product. Using the block-diagonalization in Lemma~\ref{LemmaFourierBlockDiag}, the computation of $\T{C} = \T{A} * \T{B}$ is performed by first converting the spatial-domain tensors $\T{A}$ and $\T{B}$ into the Fourier-domain tensors $\widehat{\T{A}}$ and $\widehat{\T{B}}$. This step is efficiently implemented by applying the Fast Fourier Transform (FFT) along each tube of $\T{A}$ and $\T{B}$. Next, in the Fourier domain, the T-product becomes a series of independent matrix products on each frontal slice:
\begin{equation*}
\widehat{\M{C}}^{(i)} = \widehat{\M{A}}^{(i)}\widehat{\M{B}}^{(i)}, \quad \text{for } i=1, \dots, p.
\end{equation*}
Finally, $\widehat{\T{C}}$ is converted back to the spatial domain by the inverse DFT described in Lemma \ref{LemmaFourierBlockDiag},
which is efficiently implemented using the inverse Fast Fourier Transform (IFFT).

We now define other essential tensor operations  based on~\cite{ju2025geometric}. 
\begin{definition}
The \emph{identity tensor} $\T{I}_{n,p} \in \R^{n \times n \times p}$ is the tensor whose first frontal slice is the identity matrix $\M{I}_n$ and whose other frontal slices are zero matrices, such that $\T{A} * \T{I}_{n,p} = \T{A}$ for any compatible tensor $\T{A}$.
The \emph{zero tensor} $\T{O} \in \R^{n\times m\times p}$ is defined as the tensor whose frontal slices are all the zero matrix.
\end{definition}

\begin{definition}
For a tensor $\T{A}\in \C^{m\times n\times p}$, let the \emph{T-conjugate transpose} of $\T{A}$ be denoted by $\T{A}^H$.
Then, each $k$-th frontal slice of $\T{A}^H$, denoted by $(\M{A}^H)^{(k)}$, is defined as follows
$$
(\M{A}^H)^{(k)} =
\begin{cases}
(\M{A}^{(1)})^H & \text{if } k=1 \\
(\M{A}^{(p-k+2)})^H & \text{if } k \ge 2
\end{cases}
$$
This definition ensures that $\bcirc(\T{A}^H) = (\bcirc(\T{A}))^H$.
\end{definition}

\begin{definition}
A tensor $\T{A} \in \C^{n \times n \times p}$ satisfying $\T{A}^H = \T{A}$, $\T{A}$ is called \emph{T-Hermitian}.
Equivalently, $\T{A}$ is T-Hermitian if and only if its block circulant matricization, $\bcirc(\T{A})$, is a Hermitian matrix.
\end{definition}

\begin{definition}
A tensor $\T{U} \in \C^{n \times n \times p}$ is \emph{T-unitary} if it satisfies:
$$ \T{U}^H * \T{U} = \T{U} * \T{U}^H = \T{I}_{n,p}. $$
\end{definition}

\begin{definition}
A T-Hermitian tensor $\T{M} \in \C^{n \times n \times p}$ $\mathcal{M} \in \mathbb{C}^{n \times n \times p}$ is strictly \emph{T-positive-definite} if $\tr(\T{X}^H * \T{M} * \T{X}) > 0$ for every nonzero tensor $\T{X} \in \C^{n \times 1 \times p}$.
Here, the trace operator is $\tr(\T{A}) = \tr(\bcirc(\T{A}))$ and the product $\T{X}^H * \T{M} * \T{X}$ yields a $1 \times 1 \times p$ tubal scalar.
\end{definition}
Equivalently, for every nonzero $\T{X}$, the Frobenius inner product $\langle \T{X}, \T{M} * \T{X} \rangle$ is positive. 
Furthermore, a T-Hermitian tensor is T-positive-definite if and only if its block circulant matricization $\bcirc(\T{M})$ is Hermitian positive-definite, or equivalently, if all of its Fourier-domain slices $\widehat{\M{M}}^{(i)}$ are Hermitian positive-definite matrices.

\section{Main Theoretical Framework}
\label{sec:framework}

We first recall the structured matrix classes that will be lifted to the T-product setting.
Let $\M{J} \in \R^{2n \times 2n}$ be the standard symplectic unit matrix
$\M{J} = \begin{bmatrix} \M{0} & \M{I}_{n} \\ -\M{I}_{n} & \M{0} \end{bmatrix}$.
A matrix $\M{H} \in \C^{2n \times 2n}$ is \emph{Hamiltonian} if $(\M{J}\M{H})^H = \M{J}\M{H}$. 
A matrix $\M{S} \in \C^{2n \times 2n}$ is \emph{symplectic} if $\M{S}^H \M{J} \M{S} = \M{J}$. 
We now define the tensor analogue of the matrix $\M{J}$.

\begin{definition}
\label{def:t_unit_J}
The \emph{T-symplectic unit tensor}, denoted $\T{J} \in \R^{2n \times 2n \times p}$, is the tensor whose every frontal slice in the Fourier domain is the standard symplectic unit matrix $\M{J}$:
$$
    \hat{\M{J}}^{(i)} = \M{J} = \begin{bmatrix} \M{0} & \M{I}_{n} \\ -\M{I}_{n} & \M{0} \end{bmatrix} \quad \text{for } i = 1, \dots, p.
$$
\end{definition}

\begin{remark}\label{rmk:J}
The structure of $\T{J}$ in the spatial domain follows directly from its definition in the Fourier domain. 
Applying the inverse DFT reveals that its first frontal slice $\M{J}^{(1)}$ is the matrix $\M{J}$, while all other frontal slices $\M{J}^{(k)}$ for $k > 1$ are zero matrices. 
This makes $\T{J}$ a sparse tensor in the spatial domain. In addition, since
$$
\widehat{\M{J}}^{(i)}\widehat{\M{J}}^{(i)} = \M{J}^2 = -\M{I}_{2n}
\quad \text{for each } i=1,\dots,p,
$$
we also have
$$
\T{J} * \T{J} = -\T{I}_{2n,p}.
$$
\end{remark}

With this foundational structure, we can now define T-Hamiltonian and T-symplectic tensors, and then introduce the main result of this paper.

\subsection{T-Hamiltonian Tensors}
The definition of a T-Hamiltonian tensor is a natural extension of its matrix counterpart, designed to be compatible with the algebraic structure of the T-product.
By enforcing the Hamiltonian property on each Fourier-domain slice, where the T-product behaves as conventional matrix multiplication (\cite{kilmer2011factorization}), we preserve the essential spectral characteristics.

\begin{definition}
\label{def:t_hamiltonian}
A tensor $\T{H} \in \C^{2n \times 2n \times p}$ is said to be \emph{T-Hamiltonian} if it satisfies the equation $(\T{J} * \T{H})^H = \T{J} * \T{H}$.
\end{definition}

This definition implies that the tensor product $\T{K} = \T{J} * \T{H}$ is T-Hermitian. 
The main consequence of this definition is revealed in the Fourier domain.

\begin{theorem}
\label{thm:hamiltonian_equiv}
A tensor $\T{H} \in \C^{2n \times 2n \times p}$ is T-Hamiltonian if and only if each of its frontal slices in the Fourier domain, $\widehat{\M{H}}^{(i)}$ for $i=1, \dots, p$, is a Hamiltonian matrix.
\end{theorem}
\begin{proof}
Set
$\T{K} = \T{J} * \T{H}$.
By Definition~\ref{def:t_hamiltonian}, the tensor $\T{H}$ is T-Hamiltonian if and only if $\T{K}$ is T-Hermitian.

We now map to the Fourier domain. For each $i=1,\dots,p$, the $i$-th frontal slice of $\widehat{\T{K}}$ is
$$
\widehat{\M{K}}^{(i)} = \widehat{\M{J}}^{(i)}\widehat{\M{H}}^{(i)}.
$$
By Definition~\ref{def:t_unit_J}, we have $\widehat{\M{J}}^{(i)} = \M{J}$ for every $i$, and therefore
$$
\widehat{\M{K}}^{(i)} = \M{J}\widehat{\M{H}}^{(i)}.
$$

Next we characterize what $\T{K}$ being T-Hermitian means in terms of its Fourier-domain slices. By the definition of a T-Hermitian tensor, $\T{K}$ is T-Hermitian if and only if $\bcirc(\T{K})$ is Hermitian. Lemma~\ref{LemmaFourierBlockDiag} gives
$$
\bcirc(\T{K}) = (\M{F}_{p}^H \otimes \M{I}_{2n})
\diag(\widehat{\M{K}}^{(1)},\dots,\widehat{\M{K}}^{(p)})
(\M{F}_{p} \otimes \M{I}_{2n}).
$$
Since $\M{F}_{p}$ is unitary, the matrix $\bcirc(\T{K})$ is Hermitian if and only if the block diagonal matrix
$$
\diag(\widehat{\M{K}}^{(1)},\dots,\widehat{\M{K}}^{(p)})
$$
is Hermitian. The latter condition holds if and only if each diagonal block is Hermitian, that is,
$$
(\widehat{\M{K}}^{(i)})^H = \widehat{\M{K}}^{(i)}
\quad \text{for every } i=1,\dots,p.
$$

Combining these identities, we conclude that $\T{H}$ is T-Hamiltonian if and only if
$$
(\M{J}\widehat{\M{H}}^{(i)})^H = \M{J}\widehat{\M{H}}^{(i)}
\quad \text{for every } i=1,\dots,p.
$$
By the definition of a Hamiltonian matrix, this is equivalent to saying that each matrix $\widehat{\M{H}}^{(i)}$ is Hamiltonian. This proves the assertion.
\end{proof}

This theorem provides a practical way to construct or verify T-Hamiltonian tensors. 
A common representation for Hamiltonian matrices involves a specific block structure~\cite{benner2001symplectic}.

\begin{proposition}
\label{prop:T-Hamiltonian}
Let $\T{H} \in \C^{2n \times 2n \times p}$. 
$\T{H}$ is T-Hamiltonian if and only if each Fourier-domain slice $\widehat{\M{H}}^{(i)}$ can be written in the block form $\widehat{\M{H}}^{(i)} = \begin{bmatrix} \M{A}_i & \M{B}_i \\ \M{C}_i & -\M{A}_i^H \end{bmatrix}$, where $\M{A}_i, \M{B}_i, \M{C}_i \in \C^{n \times n}$ and the matrices $\M{B}_i$ and $\M{C}_i$ are Hermitian for each $i=1, \dots, p$.
\end{proposition}
\begin{proof}
Assume first that $\T{H}$ is T-Hamiltonian. By Theorem~\ref{thm:hamiltonian_equiv}, each matrix $\widehat{\M{H}}^{(i)}$ is Hamiltonian. Fix $i \in \{1,\dots,p\}$, and write
$$
\widehat{\M{H}}^{(i)} =
\begin{bmatrix}
\M{A}_i & \M{B}_i \\
\M{C}_i & \M{D}_i
\end{bmatrix},
$$
where $\M{A}_i,\M{B}_i,\M{C}_i,\M{D}_i \in \C^{n\times n}$. We compute
$$
\M{J}\widehat{\M{H}}^{(i)}
=
\begin{bmatrix}
\M{0} & \M{I}_n \\
-\M{I}_n & \M{0}
\end{bmatrix}
\begin{bmatrix}
\M{A}_i & \M{B}_i \\
\M{C}_i & \M{D}_i
\end{bmatrix}
=
\begin{bmatrix}
\M{C}_i & \M{D}_i \\
-\M{A}_i & -\M{B}_i
\end{bmatrix}.
$$
Since $\widehat{\M{H}}^{(i)}$ is Hamiltonian, the matrix $\M{J}\widehat{\M{H}}^{(i)}$ is Hermitian. Hence
$$
\left(
\begin{bmatrix}
\M{C}_i & \M{D}_i \\
-\M{A}_i & -\M{B}_i
\end{bmatrix}
\right)^H
=
\begin{bmatrix}
\M{C}_i & \M{D}_i \\
-\M{A}_i & -\M{B}_i
\end{bmatrix}.
$$
Computing the conjugate transpose on the left-hand side gives
$$
\begin{bmatrix}
\M{C}_i^H & -\M{A}_i^H \\
\M{D}_i^H & -\M{B}_i^H
\end{bmatrix}
=
\begin{bmatrix}
\M{C}_i & \M{D}_i \\
-\M{A}_i & -\M{B}_i
\end{bmatrix}.
$$
Equality of the corresponding blocks yields
$$
\M{C}_i^H = \M{C}_i, \qquad
-\M{A}_i^H = \M{D}_i, \qquad
\M{D}_i^H = -\M{A}_i, \qquad
\M{B}_i^H = \M{B}_i.
$$
The second identity gives $\M{D}_i = -\M{A}_i^H$, and the first and fourth identities show that $\M{C}_i$ and $\M{B}_i$ are Hermitian. Therefore
$$
\widehat{\M{H}}^{(i)} =
\begin{bmatrix}
\M{A}_i & \M{B}_i \\
\M{C}_i & -\M{A}_i^H
\end{bmatrix},
$$
with $\M{B}_i^H = \M{B}_i$ and $\M{C}_i^H = \M{C}_i$.

Conversely, assume that for each $i=1,\dots,p$ we have
$$
\widehat{\M{H}}^{(i)} =
\begin{bmatrix}
\M{A}_i & \M{B}_i \\
\M{C}_i & -\M{A}_i^H
\end{bmatrix},
$$
where $\M{B}_i^H = \M{B}_i$ and $\M{C}_i^H = \M{C}_i$. We compute
$$
\M{J}\widehat{\M{H}}^{(i)}
=
\begin{bmatrix}
\M{0} & \M{I}_n \\
-\M{I}_n & \M{0}
\end{bmatrix}
\begin{bmatrix}
\M{A}_i & \M{B}_i \\
\M{C}_i & -\M{A}_i^H
\end{bmatrix}
=
\begin{bmatrix}
\M{C}_i & -\M{A}_i^H \\
-\M{A}_i & -\M{B}_i
\end{bmatrix}.
$$
Taking the conjugate transpose gives
$$
\left(
\begin{bmatrix}
\M{C}_i & -\M{A}_i^H \\
-\M{A}_i & -\M{B}_i
\end{bmatrix}
\right)^H
=
\begin{bmatrix}
\M{C}_i^H & (-\M{A}_i)^H \\
(-\M{A}_i^H)^H & (-\M{B}_i)^H
\end{bmatrix}
=
\begin{bmatrix}
\M{C}_i^H & -\M{A}_i^H \\
-\M{A}_i & -\M{B}_i^H
\end{bmatrix}.
$$
Using $\M{C}_i^H = \M{C}_i$ and $\M{B}_i^H = \M{B}_i$, we obtain
$$
(\M{J}\widehat{\M{H}}^{(i)})^H
=
\begin{bmatrix}
\M{C}_i & -\M{A}_i^H \\
-\M{A}_i & -\M{B}_i
\end{bmatrix}
=
\M{J}\widehat{\M{H}}^{(i)}.
$$
Hence $\widehat{\M{H}}^{(i)}$ is Hamiltonian for every $i$. Theorem~\ref{thm:hamiltonian_equiv} now implies that $\T{H}$ is T-Hamiltonian.
\end{proof}

The T-product framework allows for a natural extension of eigenvalue problems to tensors.
The study of tensor eigenvalues is an active field, with foundational work on variational approaches and spectral properties established for various tensor formats (\cite{qi2005eigenvalues, lim2005singular}). 
Tensor spectral theory has developed into a broad field that deals with various forms of tensor eigenvalue definitions and their properties (\cite{qi2017tensor}).
In this paper, we adopt the T-spectrum induced by the T-Jordan canonical form. By the T-Jordan decomposition theorem introduced in~\cite{miao2021t}, every tensor $\T{A} \in \C^{n \times n \times p}$ admits a factorization
\begin{equation*}
\T{A} = \T{P}^{-1} * \T{C} * \T{P},
\end{equation*}
where $\T{P}$ is invertible and $\T{C}$ is an F-upper-bi-diagonal tensor; equivalently, each Fourier-domain slice $\widehat{\M{C}}^{(i)}$ is in Jordan canonical form. This leads to the following two possible definitions for T-eigenvalue.

\begin{definition}
\label{def:jordan}
A \emph{T-eigenvalue} of $\T{A}$ is any diagonal entry of the Jordan matrices $\widehat{\M{C}}^{(i)}$, $i=1,\dots,p$.
\end{definition}

\begin{definition}
\label{def:eigen_eq}
A scalar $\lambda \in \C$ is a \emph{T-eigenvalue} of a tensor $\T{A} \in \C^{n \times n \times p}$ if there exists a nonzero tensor $\T{X} \in \C^{n \times 1 \times p}$ (a T-eigenvector) such that
$$ \T{A} * \T{X} = \lambda \T{X}. $$
\end{definition}

\begin{proposition}
\label{prop:eigen_equiv}
Definition~\ref{def:jordan} and Definition~\ref{def:eigen_eq} are mathematically equivalent.
\end{proposition}
\begin{proof}
Assume first that $\lambda$ is a T-eigenvalue in the sense of Definition~\ref{def:eigen_eq}. Then there exists a nonzero tensor $\T{X} \in \C^{n\times 1\times p}$ such that
$$
\T{A} * \T{X} = \lambda \T{X}.
$$
After mapping to the Fourier domain and using the slice-wise form of the T-product, we obtain
$$
\widehat{\M{A}}^{(i)}\widehat{\V{x}}^{(i)} = \lambda \widehat{\V{x}}^{(i)}
\quad \text{for every } i=1,\dots,p,
$$
where $\widehat{\V{x}}^{(i)}$ denotes the $i$-th frontal slice of $\widehat{\T{X}}$. Since $\T{X}\neq \T{O}$, the tensor $\widehat{\T{X}}$ is also nonzero. Hence, there exists at least one index $k$ such that $\widehat{\V{x}}^{(k)} \neq \V{0}$. For this index $k$ we have
$$
\widehat{\M{A}}^{(k)}\widehat{\V{x}}^{(k)} = \lambda \widehat{\V{x}}^{(k)},
$$
so $\lambda$ is an eigenvalue of the matrix $\widehat{\M{A}}^{(k)}$. By the ordinary Jordan theory for matrices, the eigenvalues of $\widehat{\M{A}}^{(k)}$ are exactly the diagonal entries of its Jordan canonical form $\widehat{\M{C}}^{(k)}$. Therefore $\lambda$ appears among the diagonal entries of one of the matrices $\widehat{\M{C}}^{(i)}$, and hence $\lambda$ is a T-eigenvalue in the sense of Definition~\ref{def:jordan}.

Conversely, assume that $\lambda$ is a T-eigenvalue in the sense of Definition~\ref{def:jordan}. Then there exists an index $k \in \{1,\dots,p\}$ such that $\lambda$ is a diagonal entry of $\widehat{\M{C}}^{(k)}$. Again by ordinary Jordan theory for matrices, $\lambda$ is an eigenvalue of $\widehat{\M{A}}^{(k)}$. Hence there exists a nonzero vector $\V{v} \in \C^n$ such that
$$
\widehat{\M{A}}^{(k)}\V{v} = \lambda \V{v}.
$$
We define a tensor $\widehat{\T{X}} \in \C^{n\times 1\times p}$ by prescribing its frontal slices as
$$
\widehat{\V{x}}^{(j)} =
\begin{cases}
\V{v}, & j=k,\\
\V{0}, & j\neq k.
\end{cases}
$$
The tensor $\widehat{\T{X}}$ is nonzero because $\widehat{\V{x}}^{(k)} = \V{v}\neq \V{0}$. For $j\neq k$ we have
$$
\widehat{\M{A}}^{(j)}\widehat{\V{x}}^{(j)}
=
\widehat{\M{A}}^{(j)}\V{0}
=
\V{0}
=
\lambda \V{0}
=
\lambda \widehat{\V{x}}^{(j)},
$$
and for $j=k$ we have
$$
\widehat{\M{A}}^{(k)}\widehat{\V{x}}^{(k)}
=
\widehat{\M{A}}^{(k)}\V{v}
=
\lambda \V{v}
=
\lambda \widehat{\V{x}}^{(k)}.
$$
Thus
$$
\widehat{\M{A}}^{(j)}\widehat{\V{x}}^{(j)} = \lambda \widehat{\V{x}}^{(j)}
\quad \text{for every } j=1,\dots,p.
$$
Applying the inverse DFT to map back to the spatial domain, we obtain
$$
\T{A} * \T{X} = \lambda \T{X},
$$
where $\T{X}$ is the result of applying slice-wise the inverse transform of $\widehat{\T{X}}$ as in Lemma \ref{LemmaFourierBlockDiag}. 
Since $\widehat{\T{X}}\neq \T{O}$, also $\T{X}\neq \T{O}$. Therefore $\lambda$ is a T-eigenvalue in the sense of Definition~\ref{def:eigen_eq}.
\end{proof}

The T-eigenvalues of T-Hamiltonian tensors satisfy the following symmetry.

\begin{theorem}
\label{thm:spectral_symmetry}
Let $\T{H} \in \C^{2n \times 2n \times p}$ be a T-Hamiltonian tensor. 
If $\lambda$ is a T-eigenvalue of $\T{H}$, then $-\overline{\lambda}$ is also a T-eigenvalue of $\T{H}$.
\end{theorem}
\begin{proof}
Let $\lambda$ be a T-eigenvalue of $\T{H}$. By Proposition~\ref{prop:eigen_equiv}, there exists an index $i \in \{1,\dots,p\}$ such that $\lambda$ is an eigenvalue of the matrix $\widehat{\M{H}}^{(i)}$. Set
$\M{H}_i = \widehat{\M{H}}^{(i)}$.
Since $\T{H}$ is T-Hamiltonian, Theorem~\ref{thm:hamiltonian_equiv} shows that $\M{H}_i$ is a Hamiltonian matrix, so
$$
(\M{J}\M{H}_i)^H = \M{J}\M{H}_i.
$$
Because $\M{J}^H = -\M{J}$ and $\M{J}^{-1} = -\M{J}$, this identity implies
$$
\M{H}_i^H \M{J}^H = \M{J}\M{H}_i
\quad \Longrightarrow \quad
-\M{H}_i^H \M{J} = \M{J}\M{H}_i,
$$
and, hence,
$$
\M{H}_i^H = \M{J}\M{H}_i \M{J}.
$$
Therefore,
$$
\M{J}\M{H}_i \M{J}^{-1}
=
\M{J}\M{H}_i(-\M{J})
=
-\M{J}\M{H}_i\M{J}
=
-\M{H}_i^H.
$$
Thus $\M{H}_i$ and $-\M{H}_i^H$ are similar, so they have the same spectrum
$$
\operatorname{spec}(\M{H}_i)=\operatorname{spec}(-\M{H}_i^H).
$$
Since $\operatorname{spec}(\M{H}_i^H)=\overline{\operatorname{spec}(\M{H}_i)}$, we obtain
$$
\operatorname{spec}(\M{H}_i)
=
-\overline{\operatorname{spec}(\M{H}_i)}.
$$
Because $\lambda \in \operatorname{spec}(\M{H}_i)$, it follows that $-\overline{\lambda} \in \operatorname{spec}(\M{H}_i)$. Applying Proposition~\ref{prop:eigen_equiv} once more, we conclude that $-\overline{\lambda}$ is a T-eigenvalue of $\T{H}$.
\end{proof}

\subsection{T-Symplectic Tensors}
Analogous to the T-Hamiltonian case, a T-symplectic tensor is defined to ensure that in the Fourier domain, the frontal slices are symplectic in the matrix sense.
This preserves properties
inherent to symplectic structures.

\begin{definition}
\label{def:t_symplectic}
A tensor $\T{S} \in \C^{2n \times 2n \times p}$ is said to be \emph{T-symplectic} if it satisfies the equation $\T{S}^H * \T{J} * \T{S} = \T{J}$.
\end{definition}

Moreover, from the previous definition we have the following results in the Fourier domain.

\begin{theorem}
\label{thm:symplectic_equiv}
A tensor $\T{S} \in \C^{2n \times 2n \times p}$ is T-symplectic if and only if each of its frontal slices in the Fourier domain $\widehat{\M{S}}^{(i)}$, for $i=1, \dots, p$, is a symplectic matrix.
\end{theorem}
\begin{proof}
Assume first that $\T{S}$ is T-symplectic. By Definition~\ref{def:t_symplectic},
$$
\T{S}^H * \T{J} * \T{S} = \T{J}.
$$
Mapping to the Fourier domain and using the slice-wise form of the T-product, we obtain
$$
(\widehat{\M{S}}^{(i)})^H \widehat{\M{J}}^{(i)} \widehat{\M{S}}^{(i)} = \widehat{\M{J}}^{(i)}
\quad \text{for every } i=1,\dots,p.
$$
Since $\widehat{\M{J}}^{(i)} = \M{J}$ for all $i$, the preceding identity becomes
$$
(\widehat{\M{S}}^{(i)})^H \M{J} \widehat{\M{S}}^{(i)} = \M{J}
\quad \text{for } i=1,\dots,p.
$$
Therefore each $\widehat{\M{S}}^{(i)}$ is a symplectic matrix.

Conversely, assume that each matrix $\widehat{\M{S}}^{(i)}$ is symplectic. Then,
$$
(\widehat{\M{S}}^{(i)})^H \M{J} \widehat{\M{S}}^{(i)} = \M{J}
\quad \text{for } i=1,\dots,p.
$$
Using again the identity $\widehat{\M{J}}^{(i)} = \M{J}$, we rewrite this as
$$
(\widehat{\M{S}}^{(i)})^H \widehat{\M{J}}^{(i)} \widehat{\M{S}}^{(i)} = \widehat{\M{J}}^{(i)}
\quad \text{for } i=1,\dots,p.
$$
After applying the inverse DFT, these $p$ slice-wise identities combine into
$$
\T{S}^H * \T{J} * \T{S} = \T{J}.
$$
Hence $\T{S}$ is T-symplectic.
\end{proof}

T-symplectic tensors possess several key properties that are direct extensions of their matrix counterparts. 
We begin with the inverse formula.

\begin{proposition}
If $\T{S} \in \C^{2n \times 2n \times p}$ is T-symplectic, then it is invertible, and its inverse is given by $\T{S}^{-1} = -\T{J} * \T{S}^H * \T{J}$.
\end{proposition}
\begin{proof}
Define
$$
\T{R} = -\T{J} * \T{S}^H * \T{J}.
$$
We first show that $\T{R}$ is a left inverse of $\T{S}$. Since $\T{S}$ is T-symplectic,
$$
\T{S}^H * \T{J} * \T{S} = \T{J}.
$$
Multiplying this identity on the left by $-\T{J}$, we obtain
$$
(-\T{J}) * \T{S}^H * \T{J} * \T{S} = (-\T{J}) * \T{J}.
$$
By associativity of the T-product and the identity $\T{J} * \T{J} = -\T{I}_{2n,p}$ from Remark~\ref{rmk:J}, the right-hand side becomes
$$
(-\T{J}) * \T{J} = -(\T{J} * \T{J}) = -(-\T{I}_{2n,p}) = \T{I}_{2n,p}.
$$
Hence,
$$
\T{R} * \T{S} = \T{I}_{2n,p}.
$$
Mapping to the Fourier domain gives
$$
\widehat{\M{R}}^{(i)}\widehat{\M{S}}^{(i)} = \M{I}_{2n}
\quad \text{for every } i=1,\dots,p.
$$
Therefore, each matrix $\widehat{\M{S}}^{(i)}$ is invertible and
$$
\widehat{\M{R}}^{(i)} = (\widehat{\M{S}}^{(i)})^{-1}.
$$
Consequently,
$$
\widehat{\M{S}}^{(i)}\widehat{\M{R}}^{(i)} = \M{I}_{2n}
\quad \text{for every } i=1,\dots,p.
$$
Applying the inverse DFT to this last expression, we obtain
$$
\T{S} * \T{R} = \T{I}_{2n,p}.
$$
Thus $\T{R}$ is both a left inverse and a right inverse of $\T{S}$ and, therefore,
$$
\T{S}^{-1} = \T{R} = -\T{J} * \T{S}^H * \T{J}.
$$
\end{proof}

A fundamental connection exists between Hamiltonian and symplectic structures via the matrix exponential. In the T-product setting, we define the tensor exponential by the absolutely convergent T-power series
\begin{equation*}
\exp(\T{H}) := \sum_{k=0}^{\infty} \frac{1}{k!}\T{H}^{*k}, \quad 
\T{H}^{*0} := \T{I}_{2n,p},
\end{equation*}
where $\T{H}^{*k} = \overbrace{\T{H}*\T{H}*\cdots*\T{H}}^{k~\text{times}}$, which is consistent with the general framework under the T-product in~\cite{ju2025geometric}. Equivalently, its slices in the Fourier domain are obtained by applying the matrix exponential slice-wise
\begin{equation*}
\widehat{\exp(\T{H})}^{(i)} = \exp(\widehat{\M{H}}^{(i)}),
\end{equation*}
for $i=1,\dots,p$.
This relationship extends the classical Hamiltonian-symplectic correspondence to the tensor case.

\begin{proposition}
\label{prop:exponential_map}
If $\T{H}$ is a T-Hamiltonian tensor, then $\T{S} = \exp(\T{H})$ is a T-symplectic tensor.
\end{proposition}
\begin{proof}
Let $i \in \{1,\dots,p\}$ be fixed. Then, by Theorem~\ref{thm:hamiltonian_equiv}, the matrix $\widehat{\M{H}}^{(i)}$ is Hamiltonian. Set for $t\in\mathbb{R}$
$$
\M{F}_i(t) = \exp(t\widehat{\M{H}}^{(i)})^H \M{J}\exp(t\widehat{\M{H}}^{(i)}).
$$
We differentiate $\M{F}_i(t)$ with respect to $t$. Since
$$
\frac{d}{dt}\exp(t\widehat{\M{H}}^{(i)}) = \widehat{\M{H}}^{(i)}\exp(t\widehat{\M{H}}^{(i)})
$$
and
$$
\frac{d}{dt}\exp(t\widehat{\M{H}}^{(i)})^H
=
\exp(t\widehat{\M{H}}^{(i)})^H(\widehat{\M{H}}^{(i)})^H,
$$
the product rule gives
\begin{align*}
\M{F}_i'(t)
&=
\exp(t\widehat{\M{H}}^{(i)})^H(\widehat{\M{H}}^{(i)})^H \M{J}\exp(t\widehat{\M{H}}^{(i)}) +
\exp(t\widehat{\M{H}}^{(i)})^H \M{J}\widehat{\M{H}}^{(i)}\exp(t\widehat{\M{H}}^{(i)}) \\
&=
\exp(t\widehat{\M{H}}^{(i)})^H
\left(
(\widehat{\M{H}}^{(i)})^H \M{J} + \M{J}\widehat{\M{H}}^{(i)}
\right)
\exp(t\widehat{\M{H}}^{(i)}).
\end{align*}

Since $\widehat{\M{H}}^{(i)}$ is Hamiltonian,
$$
(\M{J}\widehat{\M{H}}^{(i)})^H = \M{J}\widehat{\M{H}}^{(i)}.
$$
Because $\M{J}^H = -\M{J}$, the left-hand side is
$$
(\M{J}\widehat{\M{H}}^{(i)})^H
=
(\widehat{\M{H}}^{(i)})^H \M{J}^H
=
-(\widehat{\M{H}}^{(i)})^H \M{J}.
$$
Therefore
$$
-(\widehat{\M{H}}^{(i)})^H \M{J} = \M{J}\widehat{\M{H}}^{(i)},
$$
or equivalently,
$$
(\widehat{\M{H}}^{(i)})^H \M{J} + \M{J}\widehat{\M{H}}^{(i)} = \M{0}.
$$
Substituting this identity into the formula for $\M{F}_i'(t)$, we obtain
$$
\M{F}_i'(t) = \M{0}
\quad \text{for every } t \in \R.
$$
Hence $\M{F}_i(t)$ is constant. Evaluating at $t=0$ gives
$$
\M{F}_i(0) = \M{I}_{2n}^H \M{J}\M{I}_{2n} = \M{J},
$$
and therefore 
$$
\M{F}_i(1) = \M{J}.
$$
This means that
$$
\exp(\widehat{\M{H}}^{(i)})^H \M{J}\exp(\widehat{\M{H}}^{(i)}) = \M{J},
$$
so $\exp(\widehat{\M{H}}^{(i)})$ is symplectic.

Since this holds for every $i=1,\dots,p$, each slice of $\exp(\T{H})$ in the Fourier domain is symplectic. Theorem~\ref{thm:symplectic_equiv} now implies that $\exp(\T{H})$ is T-symplectic.
\end{proof}

\subsection{The T-Williamson Normal Form}
We state the tensor analogue of Williamson's theorem (\cite{williamson1936algebraic}). The classical result is formulated for real symmetric positive-definite matrices and yields a symplectic congruence normal form in that setting  (\cite{weedbrook2012gaussian,nicacio2021williamson}). 
Accordingly, if we assume that each slice in the Fourier domain is real symmetric positive-definite, the classical construction applies on the slices.
We also discuss the failure of a direct extension to arbitrary Hermitian positive-definite slices under the convention $\M{S}^H \M{J}\M{S} = \M{J}$.

\begin{definition}
\label{def:t_diagonal}
A tensor $\T{D} \in \C^{n\times n\times p}$ is called \emph{T-diagonal} if each frontal slice $\widehat{\M{D}}^{(i)}$ of $\widehat{\T{D}}$ is a diagonal matrix.
\end{definition}

\begin{theorem}
\label{thm:t_williamson}
Let $\T{M} \in \C^{2n \times 2n \times p}$. Assume that for each $i=1,\dots,p$, the slice $\widehat{\M{M}}^{(i)}$ is a real symmetric positive-definite matrix. Then, there exist a T-symplectic tensor $\T{S}$ and a T-diagonal tensor $\T{D}$ such that
\begin{equation}
    \T{M} = \T{S}^H * \T{D} * \T{S}.
\end{equation}
Moreover, for each $i=1,\dots,p$, the $i$-th slice of $\T{D}$ in the Fourier domain has the form
\begin{equation}
    \widehat{\M{D}}^{(i)} = \begin{bmatrix} \widehat{\M{\Lambda}}^{(i)} & \M{0} \\ \M{0} & \widehat{\M{\Lambda}}^{(i)} \end{bmatrix},
\end{equation}
where $\widehat{\M{\Lambda}}^{(i)} = \diag(\lambda_1^{(i)},\dots,\lambda_n^{(i)})$ and $\lambda_j^{(i)} > 0$ for $j=1,\dots,n$.
\end{theorem}
\begin{proof}
Let $i \in \{1,\dots,p\}$ be fixed. 
By hypothesis, the matrix $\widehat{\M{M}}^{(i)}$ is real symmetric positive-definite. 
The classical Williamson theorem for real symmetric positive-definite matrices therefore yields a real symplectic matrix $\widehat{\M{R}}^{(i)} \in \R^{2n\times 2n}$ and a diagonal matrix
$$
\widehat{\M{\Lambda}}^{(i)} = \diag(\lambda_1^{(i)},\dots,\lambda_n^{(i)})
$$
with $\lambda_j^{(i)} > 0$ for $j=1,\dots,n$, such that
\begin{equation}\label{eq:Matrix_Williamson}
(\widehat{\M{R}}^{(i)})^T \widehat{\M{M}}^{(i)} \widehat{\M{R}}^{(i)}
=
\begin{bmatrix}
\widehat{\M{\Lambda}}^{(i)} & \M{0} \\
\M{0} & \widehat{\M{\Lambda}}^{(i)}
\end{bmatrix}.
\end{equation}
Set
$$
\widehat{\M{D}}^{(i)}
=
\begin{bmatrix}
\widehat{\M{\Lambda}}^{(i)} & \M{0} \\
\M{0} & \widehat{\M{\Lambda}}^{(i)}
\end{bmatrix}
$$
and define
$$
\widehat{\M{S}}^{(i)} = (\widehat{\M{R}}^{(i)})^{-1}.
$$
Since $\widehat{\M{R}}^{(i)}$ is symplectic, we have
$$
(\widehat{\M{R}}^{(i)})^T \M{J}\widehat{\M{R}}^{(i)} = \M{J}.
$$
Multiplying this identity on the left by $(\widehat{\M{R}}^{(i)})^{-T}$ and on the right by $(\widehat{\M{R}}^{(i)})^{-1}$ yields
$$
((\widehat{\M{R}}^{(i)})^{-1})^T \M{J}(\widehat{\M{R}}^{(i)})^{-1} = \M{J},
$$
so $\widehat{\M{S}}^{(i)}$ is symplectic. Because $\widehat{\M{S}}^{(i)}$ is real, its transpose and conjugate transpose coincide. Multiplying the Williamson identity in \eqref{eq:Matrix_Williamson} on the left by $(\widehat{\M{R}}^{(i)})^{-T}$ and on the right by $(\widehat{\M{R}}^{(i)})^{-1}$, we obtain
$$
\widehat{\M{M}}^{(i)}
=
(\widehat{\M{R}}^{(i)})^{-T}\widehat{\M{D}}^{(i)}(\widehat{\M{R}}^{(i)})^{-1}
=
(\widehat{\M{S}}^{(i)})^T\widehat{\M{D}}^{(i)}\widehat{\M{S}}^{(i)}
=
(\widehat{\M{S}}^{(i)})^H\widehat{\M{D}}^{(i)}\widehat{\M{S}}^{(i)}.
$$
Thus, for the fixed index $i$, we have produced a real symplectic matrix $\widehat{\M{S}}^{(i)}$ and a matrix $\widehat{\M{D}}^{(i)}$ of the required diagonal form satisfying
$$
\widehat{\M{M}}^{(i)} = (\widehat{\M{S}}^{(i)})^H\widehat{\M{D}}^{(i)}\widehat{\M{S}}^{(i)}.
$$

We perform this construction for each $i=1,\dots,p$. Let $\widehat{\T{S}}$ and $\widehat{\T{D}}$ be the tensors whose $i$-th frontal slices are $\widehat{\M{S}}^{(i)}$ and $\widehat{\M{D}}^{(i)}$, respectively, and let $\T{S}$ and $\T{D}$ be the results of applying slice-wise the inverse transform of $\widehat{\T{S}}$ and $\widehat{\T{D}}$ as in Lemma \ref{LemmaFourierBlockDiag}. Since every frontal slice of $\widehat{\T{D}}$ is diagonal, Definition~\ref{def:t_diagonal} implies that $\T{D}$ is T-diagonal.

It remains to verify that $\T{S}$ is T-symplectic. For each $i=1,\dots,p$, the matrix $\widehat{\M{S}}^{(i)}$ is symplectic, and hence
$$
(\widehat{\M{S}}^{(i)})^H \M{J}\widehat{\M{S}}^{(i)} = \M{J}.
$$
Using Definition~\ref{def:t_unit_J}, we have $\widehat{\M{J}}^{(i)} = \M{J}$ for every $i$, so the preceding identity can be rewritten as
$$
(\widehat{\M{S}}^{(i)})^H \widehat{\M{J}}^{(i)} \widehat{\M{S}}^{(i)} = \widehat{\M{J}}^{(i)}
\quad \text{for every } i=1,\dots,p.
$$
By Theorem~\ref{thm:symplectic_equiv}, this implies that $\T{S}$ is T-symplectic.

Define now
$$
\T{N} := \T{S}^H * \T{D} * \T{S}.
$$
For each $i=1,\dots,p$, the $i$-th frontal slice of $\widehat{\T{N}}$ is
$$
\widehat{\M{N}}^{(i)}
=
(\widehat{\M{S}}^{(i)})^H \widehat{\M{D}}^{(i)} \widehat{\M{S}}^{(i)}
=
\widehat{\M{M}}^{(i)}.
$$
Thus, the tensors $\widehat{\T{N}}$ and $\widehat{\T{M}}$ have identical frontal slices. Applying the inverse DFT yields
$$
\T{M} = \T{S}^H * \T{D} * \T{S}.
$$
\end{proof}

\begin{algorithm}[t]
\caption{T-Williamson normal form}
\label{alg:t_williamson}
\small
\DontPrintSemicolon
\SetKwInOut{Input}{Input}
\SetKwInOut{Output}{Output}
\SetKwFunction{Williamson}{Williamson}
\SetKwFunction{DFT}{DFT}
\SetKwFunction{IDFT}{IDFT}
\Input{$\T{M}\in\C^{2n\times 2n\times p}$ with real symmetric positive-definite Fourier-domain slices}
\Output{Tensors $\T{S}$ and $\T{D}$ such that $\T{M}=\T{S}^H*\T{D}*\T{S}$}
$\widehat{\T{M}}\leftarrow \DFT_3(\T{M})$\;
\For{$i=1,\ldots,p$}{
$(\widehat{\M{R}}^{(i)},\widehat{\M{\Lambda}}^{(i)})\leftarrow \Williamson(\widehat{\M{M}}^{(i)})$\;
$\widehat{\M{D}}^{(i)}\leftarrow \operatorname{blkdiag}(\widehat{\M{\Lambda}}^{(i)},\widehat{\M{\Lambda}}^{(i)})$\;
$\widehat{\M{S}}^{(i)}\leftarrow(\widehat{\M{R}}^{(i)})^{-1}$\;
}
$\T{S}\leftarrow \IDFT_3(\widehat{\T{S}})$\;
$\T{D}\leftarrow \IDFT_3(\widehat{\T{D}})$\;
\Return{$\T{S},\T{D}$}\;
\end{algorithm}

The constructive proof gives the slice-wise 
Algorithm \ref{alg:t_williamson}, where
$\Williamson(\M{A})$ denotes the classical Williamson decomposition (\cite{williamson1936algebraic,nicacio2021williamson,houde2024matrix}), returning a real symplectic matrix $\M{R}$ and a diagonal matrix $\M{\Lambda}$ such that $\M{R}^T\M{A}\M{R}=\operatorname{blkdiag}(\M{\Lambda},\M{\Lambda})$.

The next proposition shows that the real-slice hypothesis is essential under the Hermitian symplectic convention adopted here (\cite{simon1999congruences,ding2024quantum}).

\begin{proposition}
\label{prop:complex_obstruction}
Under the convention $\M{S}^H \M{J}\M{S} = \M{J}$, a Williamson-type factorization of the form
$$
\M{M} = \M{S}^H
\begin{bmatrix}
\M{\Lambda} & \M{0} \\
\M{0} & \M{\Lambda}
\end{bmatrix}
\M{S},
\qquad
\M{\Lambda} = \diag(\lambda_1,\dots,\lambda_n), \quad \lambda_j > 0,
$$
does not exist for every complex Hermitian positive-definite matrix $\M{M} \in \C^{2n \times 2n}$.
\end{proposition}
\begin{proof}
We provide a counterexample. Consider the Hermitian positive-definite matrix
$$
\M{M}_0 =
\begin{bmatrix}
2 & \sqrt{-1} \\
-\sqrt{-1} & 2
\end{bmatrix}.
$$
Suppose, for contradiction, that there exist a matrix $\M{S} \in \C^{2 \times 2}$ and a positive number $\lambda$ such that
\begin{equation}\label{eq:M0_decomposition}
\M{M}_0 = \M{S}^H
\begin{bmatrix}
\lambda & 0 \\
0 & \lambda
\end{bmatrix}
\M{S}
\end{equation}
and
$$
\M{S}^H \M{J}\M{S} = \M{J}.
$$
Set
$$
\M{D} =
\begin{bmatrix}
\lambda & 0 \\
0 & \lambda
\end{bmatrix}.
$$
Since $\M{M}_0 = \M{S}^H \M{D}\M{S}$ and $\M{S}$ is invertible, we have
$$
\M{M}_0^{-1} = \M{S}^{-1}\M{D}^{-1}\M{S}^{-H}.
$$
Multiplying by $\M{J}$ on the right yields
$$
\M{M}_0^{-1}\M{J} = \M{S}^{-1}\M{D}^{-1}\M{S}^{-H}\M{J}.
$$
From the symplectic identity $\M{S}^H \M{J}\M{S} = \M{J}$, multiplying on the left by $\M{S}^{-H}$ gives
$$
\M{J}\M{S} = \M{S}^{-H}\M{J}.
$$
Substituting this relation into the previous identity, we obtain
$$
\M{M}_0^{-1}\M{J} = \M{S}^{-1}\M{D}^{-1}\M{J}\M{S}.
$$
Hence, $\M{M}_0^{-1}\M{J}$ is similar to $\M{D}^{-1}\M{J}$.

Now,
$$
\M{D}^{-1} =
\begin{bmatrix}
\lambda^{-1} & 0 \\
0 & \lambda^{-1}
\end{bmatrix},
$$
and therefore
$$
\M{D}^{-1}\M{J}
=
\begin{bmatrix}
\lambda^{-1} & 0 \\
0 & \lambda^{-1}
\end{bmatrix}
\begin{bmatrix}
0 & 1 \\
-1 & 0
\end{bmatrix}
=
\begin{bmatrix}
0 & \lambda^{-1} \\
-\lambda^{-1} & 0
\end{bmatrix}.
$$
The characteristic polynomial of this matrix is
$$
\det
\begin{bmatrix}
\mu & -\lambda^{-1} \\
\lambda^{-1} & \mu
\end{bmatrix}
=
\mu^2 + \lambda^{-2},
$$
so its eigenvalues are $\pm \sqrt{-1}\,\lambda^{-1}$. Consequently, the spectrum of $\M{D}^{-1}\M{J}$, and hence also the spectrum of $\M{M}_0^{-1}\M{J}$, must be
\begin{equation}\label{eq:spectrum}
\left\{
\pm \frac{\sqrt{-1}}{\lambda}
\right\}.
\end{equation}

On the other hand, for the explicit matrix $\M{M}_0$, we compute
$$
\M{M}_0^{-1}
=
\frac{1}{3}
\begin{bmatrix}
2 & -\sqrt{-1} \\
\sqrt{-1} & 2
\end{bmatrix},
$$
and hence
$$
\M{M}_0^{-1}\M{J}
=
\frac{1}{3}
\begin{bmatrix}
\sqrt{-1} & 2 \\
-2 & \sqrt{-1}
\end{bmatrix}.
$$
Therefore
$$
\det\left(\mu \M{I}_2 - \M{M}_0^{-1}\M{J}\right)
=
\left(\mu - \frac{\sqrt{-1}}{3}\right)^2 + \frac{4}{9}
=
(\mu - \sqrt{-1})\left(\mu + \frac{\sqrt{-1}}{3}\right).
$$
Thus the eigenvalues of $\M{M}_0^{-1}\M{J}$ are $\sqrt{-1}$ and $-\frac{\sqrt{-1}}{3}$. These two numbers are not of the form $\pm \sqrt{-1}\,\lambda^{-1}$ for a single positive number $\lambda$, which contradicts the necessary spectral form derived above in \eqref{eq:spectrum}. 
Therefore the decomposition in \eqref{eq:M0_decomposition} cannot exist for $\M{M}_0$.
\end{proof}

The following proposition identifies the necessary conjugate-symmetry condition in the Fourier domain that guarantees real-valued recovery in the spatial domain.

\begin{proposition}
\label{prop:real_recovery}
Let $\widehat{\T{A}} \in \C^{m \times n \times p}$, and let $\T{A}$ be the tensor obtained by applying the inverse DFT along the third mode, that is,
$$
\M{A}^{(k)} = \frac{1}{p}\sum_{i=1}^{p}\omega^{-(i-1)(k-1)}\widehat{\M{A}}^{(i)},
\quad k=1,\dots,p.
$$
Then, $\T{A}$ lies in $\R^{m \times n \times p}$ if and only if
$$
\widehat{\M{A}}^{(1)} = \overline{\widehat{\M{A}}^{(1)}}
$$
and
$$
\widehat{\M{A}}^{(p-i+2)} = \overline{\widehat{\M{A}}^{(i)}}
\quad \text{for every } i=2,\dots,p.
$$
\end{proposition}
\begin{proof}
Assume first that $\T{A}$ is real-valued. For each $i=1,\dots,p$, the forward DFT formula in Lemma \ref{LemmaFourierBlockDiag} gives
$$
\widehat{\M{A}}^{(i)} = \sum_{j=1}^{p}\omega^{(i-1)(j-1)}\M{A}^{(j)}.
$$
Since every $\M{A}^{(j)}$ is real, we have
$$
\overline{\widehat{\M{A}}^{(i)}}
=
\sum_{j=1}^{p}\overline{\omega^{(i-1)(j-1)}}\,\M{A}^{(j)}
=
\sum_{j=1}^{p}\omega^{-(i-1)(j-1)}\M{A}^{(j)}.
$$
For $i=1$, this becomes
$$
\overline{\widehat{\M{A}}^{(1)}}
=
\sum_{j=1}^{p}\omega^{0}\M{A}^{(j)}
=
\widehat{\M{A}}^{(1)}.
$$
Now let $i \in \{2,\dots,p\}$. Since $\omega^{p}=1$, we have
$$
\omega^{-(i-1)(j-1)}
=
\omega^{(p-i+1)(j-1)}
=
\omega^{(p-i+2-1)(j-1)}.
$$
Therefore
$$
\overline{\widehat{\M{A}}^{(i)}}
=
\sum_{j=1}^{p}\omega^{(p-i+2-1)(j-1)}\M{A}^{(j)}
=
\widehat{\M{A}}^{(p-i+2)}.
$$
This proves the stated conjugate-symmetry condition.

Conversely, assume that
$$
\widehat{\M{A}}^{(1)} = \overline{\widehat{\M{A}}^{(1)}}
$$
and
$$
\widehat{\M{A}}^{(p-i+2)} = \overline{\widehat{\M{A}}^{(i)}}
\quad \text{for every } i=2,\dots,p.
$$
Fix $k \in \{1,\dots,p\}$. Taking the complex conjugate of the inverse DFT formula gives
$$
\overline{\M{A}^{(j)}}
=
\frac{1}{p}\sum_{i=1}^{p}\omega^{(i-1)(j-1)}\overline{\widehat{\M{A}}^{(i)}}.
$$
We separate the term $i=1$ and use the assumption on $\widehat{\M{A}}^{(1)}$:
$$
\overline{\M{A}^{(j)}}
=
\frac{1}{p}\widehat{\M{A}}^{(1)}
+
\frac{1}{p}\sum_{i=2}^{p}\omega^{(i-1)(j-1)}\overline{\widehat{\M{A}}^{(i)}}.
$$
In the remaining sum, set $k = p-i+2$.
Then, as $i$ runs from $2$ to $p$, the index $k$ also runs through the set $\{2,\dots,p\}$. 
Moreover, $i-1 = p-k+1$.
Hence,
$$
\omega^{(i-1)(j-1)}
=
\omega^{(p-k+1)(j-1)}
=
\omega^{p(j-1)}\omega^{-(k-1)(j-1)}
=
\omega^{-(k-1)(j-1)},
$$
because $\omega^{p(j-1)} = 1$. Using the conjugate-symmetry assumption, we obtain
\begin{align*}
\overline{\M{A}^{(j)}}
&=
\frac{1}{p}\widehat{\M{A}}^{(1)}
+
\frac{1}{p}\sum_{k=2}^{p}\omega^{-(k-1)(j-1)}\widehat{\M{A}}^{(j)} \\
&=
\frac{1}{p}\sum_{k=1}^{p}\omega^{-(k-1)(j-1)}\widehat{\M{A}}^{(k)} \\
&=
\M{A}^{(j)}.
\end{align*}
Thus, every frontal slice $\M{A}^{(j)}$ is real, and therefore $\T{A} \in \R^{m \times n \times p}$.
\end{proof}

\begin{corollary}
    Assume that $\T{M}$ satisfies the hypotheses of Theorem~\ref{thm:t_williamson} and that its slices in the Fourier domain satisfy this conjugate-symmetry condition. Then the T-Williamson decomposition in Theorem~\ref{thm:t_williamson} may be chosen so that the tensors $\T{S}$ and $\T{D}$ are real-valued.
\end{corollary}
\begin{proof}
Assume that $\T{M}$ satisfies the hypotheses of Theorem~\ref{thm:t_williamson} and the above conjugate-symmetry condition. 
We enforce the Williamson factors in the Fourier domain to satisfy the conjugate-symmetric condition of Proposition~\ref{prop:real_recovery}.
First consider the self-conjugate slices. The condition for $i=1$ gives
$$
\widehat{\M{M}}^{(1)} = \overline{\widehat{\M{M}}^{(1)}},
$$
so $\widehat{\M{M}}^{(1)}$ is a real symmetric positive-definite matrix. If $p$ is even, then the index $i=\frac{p}{2}+1$ satisfies $p-i+2=i$, and hence
$$
\widehat{\M{M}}^{(\frac{p}{2}+1)} = \overline{\widehat{\M{M}}^{(\frac{p}{2}+1)}},
$$
so this slice is also real symmetric positive-definite. For these self-conjugate slices we choose real symplectic Williamson factors and real diagonal factors.

Now let $i \in \{2,\dots,p\}$ with $p-i+2 \neq i$. We apply the classical Williamson factorization to the real symmetric positive-definite matrix $\widehat{\M{M}}^{(i)}$, exactly as in the slice-wise construction used in the proof of Theorem~\ref{thm:t_williamson}, and choose matrices $\widehat{\M{S}}^{(i)}$ and $\widehat{\M{D}}^{(i)}$ such that
$$
\widehat{\M{M}}^{(i)} = (\widehat{\M{S}}^{(i)})^H \widehat{\M{D}}^{(i)} \widehat{\M{S}}^{(i)}.
$$
Because the diagonal entries of $\widehat{\M{D}}^{(i)}$ are positive real numbers, the matrix $\widehat{\M{D}}^{(i)}$ is real. We then define
$$
\widehat{\M{S}}^{(p-i+2)} = \overline{\widehat{\M{S}}^{(i)}},
\qquad
\widehat{\M{D}}^{(p-i+2)} = \widehat{\M{D}}^{(i)}.
$$
Since $\M{J}$ is real, the symplectic identity is preserved under complex conjugation:
$$
(\widehat{\M{S}}^{(p-i+2)})^H \M{J}\widehat{\M{S}}^{(p-i+2)}
=
(\overline{\widehat{\M{S}}^{(i)}})^H \M{J}\,\overline{\widehat{\M{S}}^{(i)}}
=
\overline{(\widehat{\M{S}}^{(i)})^H \M{J}\widehat{\M{S}}^{(i)}}
=
\M{J}.
$$
Moreover,
\begin{align*}
(\widehat{\M{S}}^{(p-i+2)})^H \widehat{\M{D}}^{(p-i+2)} \widehat{\M{S}}^{(p-i+2)}
&=
(\overline{\widehat{\M{S}}^{(i)}})^H \widehat{\M{D}}^{(i)} \overline{\widehat{\M{S}}^{(i)}} \\
&=
\overline{(\widehat{\M{S}}^{(i)})^H \widehat{\M{D}}^{(i)} \widehat{\M{S}}^{(i)}} \\
&=
\overline{\widehat{\M{M}}^{(i)}} \\
&=
\widehat{\M{M}}^{(p-i+2)}.
\end{align*}
Thus, the factors can be chosen so that the  slices of $\widehat{\T{S}}$ and $\widehat{\T{D}}$ in the Fourier domain satisfy the same conjugate-symmetry condition as $\widehat{\T{M}}$. Applying Proposition \ref{prop:real_recovery} to $\widehat{\T{S}}$ and $\widehat{\T{D}}$, we conclude that the tensors $\T{S}$ and $\T{D}$, obtained by applying the inverse DFT, are real-valued.
\end{proof}

\section{Numerical Validation and Experiments}
\label{sec:application}

Section~\ref{sec:framework} developed the Hamiltonian-symplectic theory in the T-product algebra. 
This section examines its numerical consequences for parameter-dependent structured matrix families. 
The experiments verify the defining relations and constructive decomposition, assess runtime trends, and conclude with a Fourier-domain encoding of sampled covariance matrices from continuous-variable quantum dynamics.
All algorithms are implemented in MATLAB R2025b. 
Numerical errors are evaluated against a tolerance of $10^{-10}$.

\subsection{Verification of Theoretical Properties}

We first check the identities used in the preceding section. 
The T-Hamiltonian tensors are randomly generated by
enforcing Fourier-domain slices to be Hamiltonian,
as in Proposition~\ref{prop:T-Hamiltonian}. 
For the Williamson test, the Fourier-domain slices of $\T{M}$ are generated as real symmetric positive-definite matrices.

Let
\[
\T{K}=\T{J}*\T{H},\qquad
\T{S}_{\exp}=\exp(\T{H}).
\]
For the Williamson test, let
\[
\T{M}=\T{S}_{\mathrm{W}}^H*\T{D}*\T{S}_{\mathrm{W}}
\]
be the computed T-Williamson decomposition. The diagonal tensor $\T{D}$ has Fourier-domain diagonal entries
\[
\widehat{\M{D}}^{(i)}
=
\operatorname{diag}(\lambda^{(i)},\lambda^{(i)}),
\]
while $\mu^{(i)}$ denotes the symplectic eigenvalues of
$\widehat{\M{M}}^{(i)}$ computed independently. 
We use the following residuals to evaluate the theoretical concepts:
\begin{align*}
\mathrm{res}_{\mathrm{Symp}}
&=
\frac{\|\T{S}_{\exp}^H*\T{J}*\T{S}_{\exp}-\T{J}\|}
{\|\T{J}\|},\\
\mathrm{res}_{\mathrm{WNF}}
&=
\frac{\|\T{M}-\T{S}_{\mathrm{W}}^H*\T{D}*\T{S}_{\mathrm{W}}\|}
{\|\T{M}\|},\\
\mathrm{res}_{\mathrm{WSp}}
&=
\frac{\|\T{S}_{\mathrm{W}}^H*\T{J}*\T{S}_{\mathrm{W}}-\T{J}\|}
{\|\T{J}\|},\\
\mathrm{res}_{\mathrm{Spec}}
&=
\frac{
\left(\sum_{i=1}^{p}
\|\lambda^{(i)}-\mu^{(i)}\|_2^2
\right)^{1/2}}
{
\left(\sum_{i=1}^{p}
\|\mu^{(i)}\|_2^2
\right)^{1/2}}.
\end{align*}
Here, $\mathrm{res}_{\mathrm{Symp}}$ measures the exponential-map construction, $\mathrm{res}_{\mathrm{WNF}}$ measures the reconstruction identity, $\mathrm{res}_{\mathrm{WSp}}$ checks that the Williamson factor satisfies the T-symplectic identity, and $\mathrm{res}_{\mathrm{Spec}}$ checks the diagonal entries against an independent computation of the symplectic spectrum.

The results are shown in Table~\ref{tab:consistency_residuals}. 
The Hamiltonian residual is not listed, because it is zero by construction: each Fourier-domain slice is generated in Hamiltonian block form.

\begin{table}[ht!]
\centering
\small
\setlength{\tabcolsep}{4pt}
\caption{Residuals for consistency checks.}
\label{tab:consistency_residuals}
\begin{tabular}{@{}cccccc@{}}
\toprule
$n$ & $p$ 
& $\mathrm{res}_{\mathrm{Symp}}$ 
& $\mathrm{res}_{\mathrm{WNF}}$ 
& $\mathrm{res}_{\mathrm{WSp}}$ 
& $\mathrm{res}_{\mathrm{Spec}}$ \\
\midrule
$4$  & $8$  & $4.348\times 10^{-14}$ & $2.313\times 10^{-15}$ & $9.455\times 10^{-16}$ & $9.348\times 10^{-16}$ \\
$8$  & $16$ & $2.186\times 10^{-12}$ & $2.879\times 10^{-15}$ & $1.329\times 10^{-15}$ & $1.768\times 10^{-15}$ \\
$12$ & $24$ & $1.417\times 10^{-11}$ & $3.628\times 10^{-15}$ & $1.734\times 10^{-15}$ & $2.093\times 10^{-15}$ \\
$16$ & $32$ & $7.260\times 10^{-11}$ & $3.752\times 10^{-15}$ & $2.029\times 10^{-15}$ & $2.223\times 10^{-15}$ \\
\bottomrule
\end{tabular}
\end{table}

The following example illustrates the spectral symmetry asserted in Theorem~\ref{thm:spectral_symmetry}.

\begin{example}
Let $\T{H}\in \C^{16 \times 16 \times 16}$. We compute the eigenvalues of each of its Fourier-domain slices $\widehat{\M{H}}^{(i)}$. 
The union of these eigenvalues forms the T-spectrum of $\T{H}$. 
As shown in Figure \ref{fig:spectrum}, the T-eigenvalues on the complex plane are symmetric with respect to the imaginary axis, visually confirming that if $\lambda$ is a T-eigenvalue, so is $-\overline{\lambda}$.

\begin{figure}[ht!]
    \centering
    \includegraphics[width=0.7\textwidth]{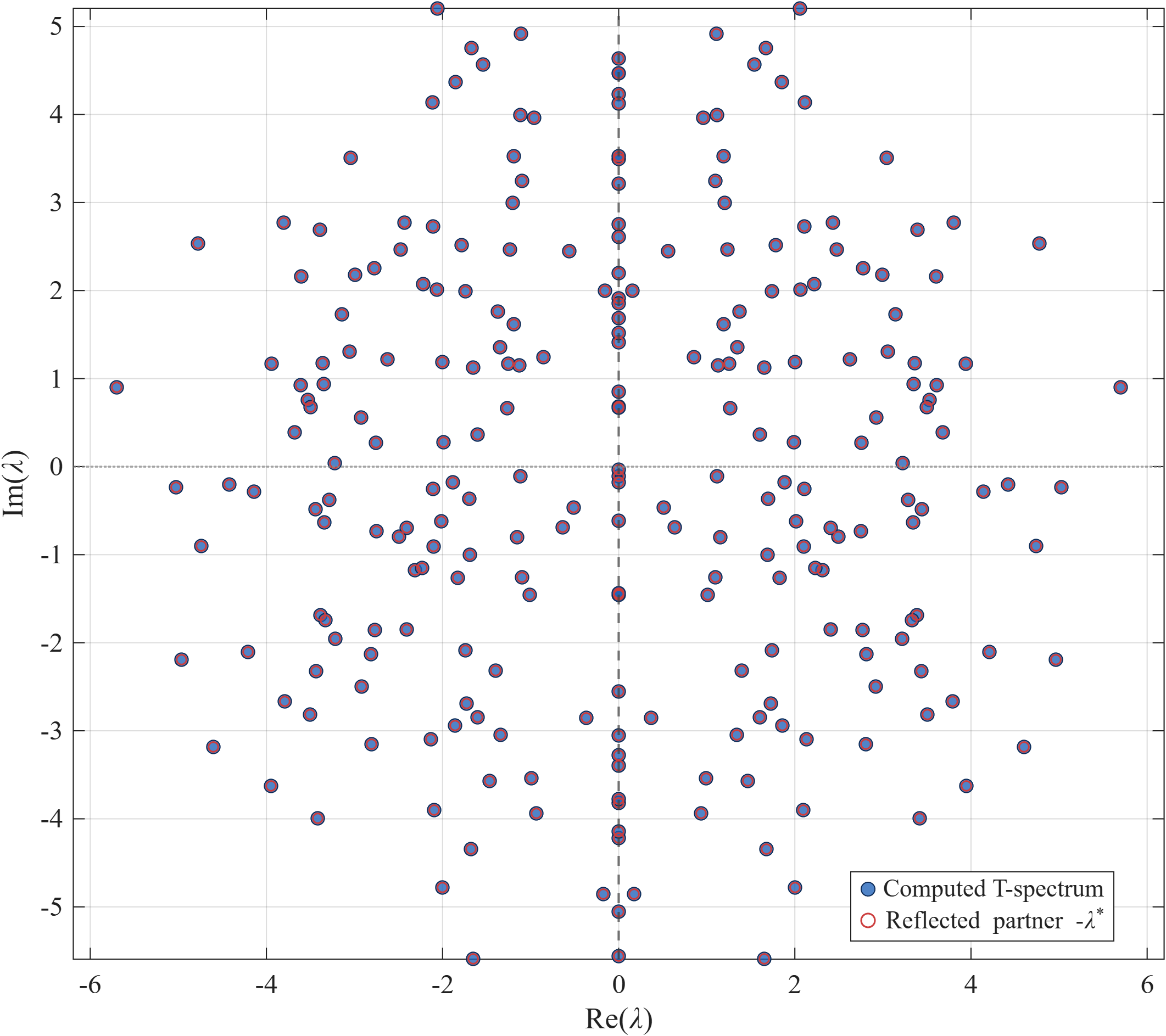}
    \caption{T-eigenvalue spectrum of a randomly generated $16 \times 16 \times 16$ T-Hamiltonian tensor. Blue markers denote the computed T-spectrum, and red hollow markers denote the reflected values $-\overline{\lambda}$. The two sets coincide pairwise, illustrating the spectral symmetry asserted in Theorem~\ref{thm:spectral_symmetry}.}
    \label{fig:spectrum}
\end{figure}
\end{example}

\subsection{Computational Cost and Scalability Analysis}
To assess the practical efficiency of our framework, we analyze the computational cost of the T-Williamson normal form algorithm (Theorem~\ref{thm:t_williamson}). 
This example reports the corresponding scalability behavior.

\begin{example}
The arithmetic cost is governed by the slice-wise classical Williamson decomposition applied to the $p$ frontal slices in the Fourier domain. 
For each slice, the constructive procedure involves dense matrix operations such as the matrix square root, the real Schur decomposition, and diagonal rescaling; each of these operations has cubic arithmetic complexity in the matrix-size parameter $n$. 
Accordingly, the arithmetic complexity of the full T-Williamson algorithm is $O(pn^3)$. 

To assess the empirical computational scaling of the implementation, we report the average of the execution time over 20 independent runs.
For the sweep in $n$, we use the sampled values $n \in \{32,48,64,96,128,160\}$ with fixed $p=16$.
The results are displayed in Figure~\ref{fig:cost_analysis}.
The left plot shows the measured runtime together with a cubic reference curve scaled to the measured timings, reflecting the slice-wise arithmetic cost $O(pn^3)$. 
The right plot shows the measured runtime as a function of the number of Fourier-domain slices $p$ for fixed $n=8$, together with a linear reference curve.

\begin{figure}[ht!]
    \centering
    \includegraphics[width=0.49\textwidth]{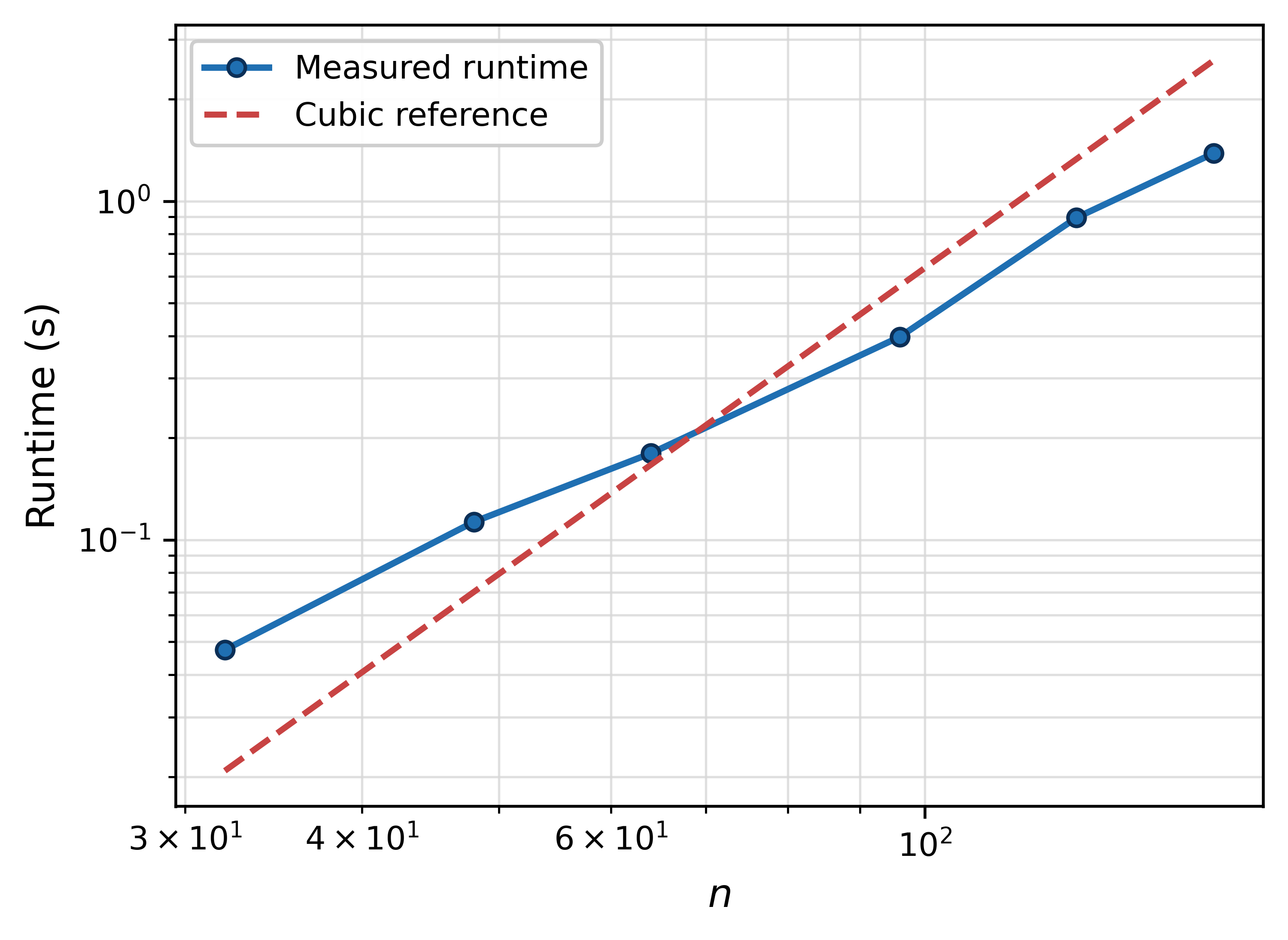}
    \hfill
    \includegraphics[width=0.49\textwidth]{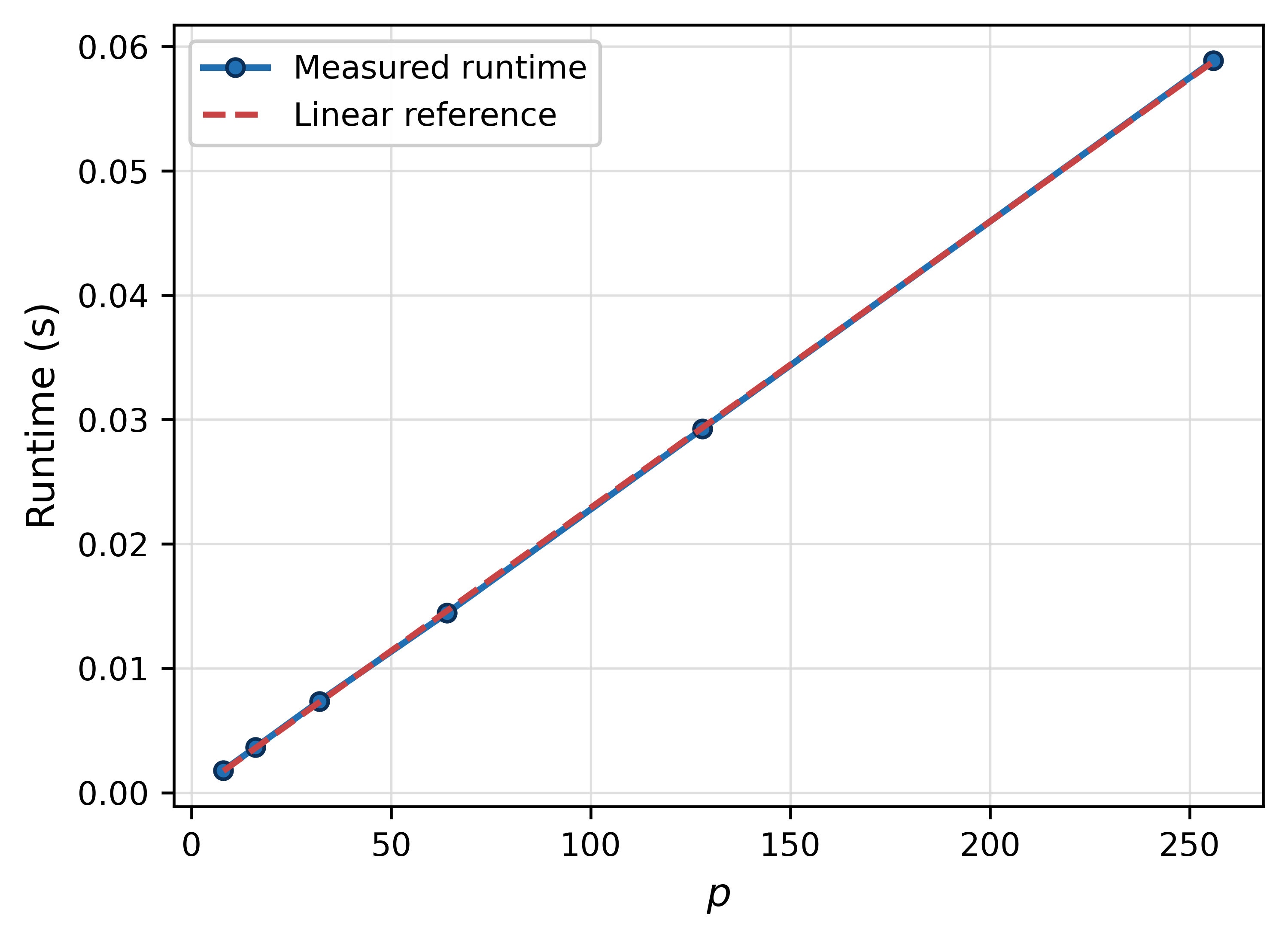}
    \caption{Runtime behavior of the T-Williamson normal form. Average runtime as $n$ for fixed $p=16$, with a cubic reference curve (left). 
    Average runtime as $p$ for fixed $n=8$, with a linear reference curve (right).}
    \label{fig:cost_analysis}
\end{figure}
\end{example}

\subsection{Decohering Quantum System}

The T-Williamson normal form can be used to organize a family of sampled covariance matrices once it is encoded in to the slices of a tensor in the Fourier domain.  We illustrate this with a two-mode squeezed vacuum state interacting with a thermal environment, while keeping the physical interpretation at the level of the slices. 
We emphasize that this construction is a Fourier-domain encoding. 
The sampled covariance matrices are not obtained by applying the DFT to a real-valued tensor. 
Instead, they are defined directly as the Fourier-domain slices of a tensor. 
This choice is consistent with the hypotheses of Theorem~\ref{thm:t_williamson}, which require the Fourier-domain slices to be real symmetric positive-definite. 
If the sampled covariance matrices were assigned to the spatial-domain frontal slices of a real tensor and the DFT were then applied, the resulting Fourier-domain slices would generally be complex Hermitian rather than real symmetric. 
In that case, Theorem~\ref{thm:t_williamson} would not be applicable.

We take $n=2$, $p=64$, and set the physical parameters for the dynamics: $r=1$ is the squeezing parameter of the initial state, $\bar{n}_{\text{th}}=0.5$ is the mean thermal photon number of the environment, and $\kappa=0.3$ is the decay rate governing the system-bath interaction.
In the phase-space ordering
$$
[\hat{q}_1,\hat{q}_2,\hat{p}_1,\hat{p}_2]^T,
$$
where the pair $(\hat{q}_j,\hat{p}_j)$ refers to the position-like component and momentum-like component of the mode $j$, the initial covariance matrix of the two-mode squeezed vacuum state (\cite{weedbrook2012gaussian}) is
$$
\M{M}_0 = \begin{bmatrix}
\cosh(2r) & \sinh(2r) & 0 & 0 \\
\sinh(2r) & \cosh(2r) & 0 & 0 \\
0 & 0 & \cosh(2r) & -\sinh(2r) \\
0 & 0 & -\sinh(2r) & \cosh(2r)
\end{bmatrix}.
$$
The interaction with a thermal bath is modeled by
$$
\M{M}(t) = e^{-\kappa t} \M{M}_0 + (1 - e^{-\kappa t}) (2\bar{n}_{\text{th}} + 1)\M{I}_4,
$$
where $(2\bar{n}_{\text{th}} + 1)\M{I}_4$ is the thermal covariance for two modes (\cite{weedbrook2012gaussian}). 
We sample the time interval $[0,12]$ at
$$
t_i = \frac{12(i-1)}{p-1},
\quad i=1,\dots,p.
$$
We encode the sampled family slice-wise for a tensor in the Fourier domain by setting
$$
\widehat{\M{M}}^{(i)} = \M{M}(t_i),
$$
for $i=1,\dots,p.$ This places the time-dependent covariance family in a single algebraic object. Since each matrix $\M{M}(t_i)$ is real symmetric positive-definite, the hypothesis of Theorem~\ref{thm:t_williamson} is satisfied slice-wise in the Fourier domain. Proposition~\ref{prop:real_recovery} identifies an additional Fourier conjugate-symmetry condition that would force the spatial-domain tensor itself to be real-valued, but that condition is not imposed here.

To analyze the physical significance of the slice-wise symplectic eigenvalues, we apply Theorem~\ref{thm:t_williamson} to decompose the covariance tensor as
$$
\T{M} = \T{S}^H * \T{D} * \T{S},
$$
where $\T{S}$ is a T-symplectic tensor and $\T{D}$ is a T-diagonal tensor. From this T-Williamson decomposition, we obtain the following diagonal matrix slices
$$
\widehat{\M{\Lambda}}^{(i)} =
\diag(\lambda_1^{(i)},\lambda_2^{(i)})
$$
for each $i=1,\dots,p$. Because the Fourier-domain slices of $\T{M}$ are exactly the covariance matrices $\M{M}(t_i)$, the diagonal entries $\lambda_1^{(i)}$ and $\lambda_2^{(i)}$ are precisely the ordinary symplectic eigenvalues of $\M{M}(t_i)$.
Under the convention $\hbar=2$, where $\hbar$ is the reduced Planck constant, the Heisenberg uncertainty principle requires
$$
\lambda_j^{(i)} \ge 1,
\quad j=1,2,
$$
for each $i=1,\dots,p$. By using the relation $\lambda_j=2\bar{n}_j+1$, given by Williamson's theorem on Gaussian states, on the entropy profile formula in~\cite{serafini2004symplectic}, we get
\begin{equation*}
    S_i = \sum_{j=1}^{n} g(\lambda_j^{(i)}),
\end{equation*}
where $g(\lambda) = \frac{\lambda + 1}{2} \log_2\left(\frac{\lambda + 1}{2}\right) - \frac{\lambda - 1}{2} \log_2\left(\frac{\lambda - 1}{2}\right)$. 
To compute the entanglement of the two pairs $(\hat{q}_j,\hat{p}_j)$, $j=1,2$, we introduce the partially transposed covariance matrix
$$
\widetilde{\M{M}}(t_i) = \M{P}\M{M}(t_i)\M{P},
\quad
\M{P} := \diag(1,1,1,-1),
$$
where $\M{P}$ sends $\hat{p}_2$ to $-\hat{p}_2$, and denote its symplectic eigenvalues by $\tilde{\lambda}_1^{(i)}$ and $\tilde{\lambda}_2^{(i)}$. The entanglement profile is then
\begin{equation*}
    E_N(\rho_i) = \sum_{j=1}^{n} \max(0, -\log_2(\tilde{\lambda}_j^{(i)})),
\end{equation*}
by ~\cite{adesso2004extremal}, where $\rho_i$ is the Gaussian state associated to the covariance matrix $\M{M}(t_i)$.

The entropy and entanglement profiles
obtained from these sampled covariance matrices are illustrated in Figure~\ref{fig:quantum_dynamics}. For the chosen parameters, the entropy increases from $0$ to approximately $2.826$ over the sampled interval, whereas the  entanglement profile
decreases from approximately $2.885$ to $0$. 
Thus, the two modes $(\hat{q}_1,\hat{p}_1)$ and $(\hat{q}_2,\hat{p}_2)$ of the system approach a thermal equilibrium and are no longer entangled after interacting with the thermal environment.
This behavior is consistent with the physical picture of decoherence: as $t$ increases, the covariance matrix approaches the thermal limit
$$
\M{M}(t) \to (2\bar{n}_{\text{th}}+1)\M{I}_4,
$$
the symplectic eigenvalues move toward the thermal value $2\bar{n}_{\text{th}}+1$, the entropy moves toward the corresponding thermal entropy, and the entanglement vanishes. 
The example shows how the T-product formalism organizes and analyzes a sampled covariance family while keeping the physical interpretation at the level of the Fourier-domain slices.

\begin{figure}[ht!]
    \centering
    \includegraphics[width=0.92\textwidth]{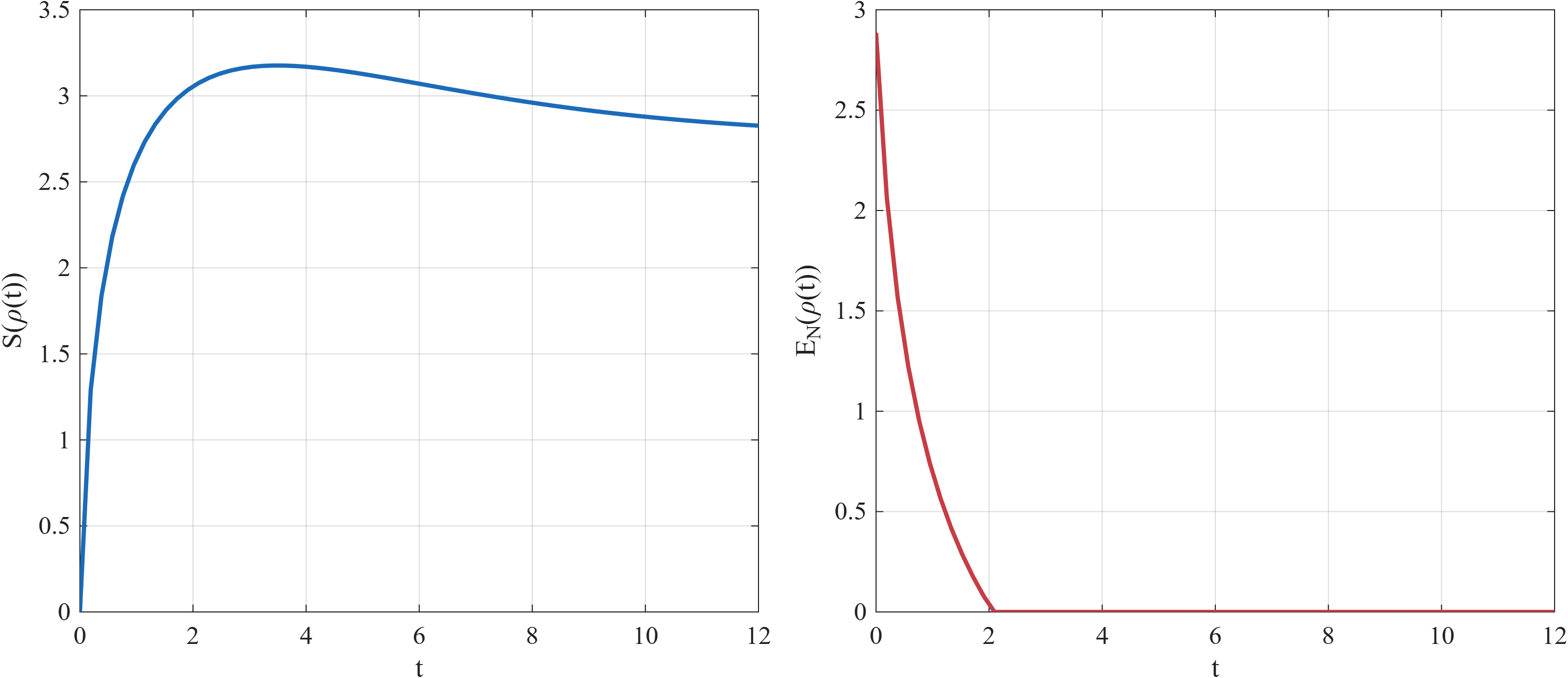}
    \caption{Quantum-dynamical profiles for the decohering two-mode squeezed state. The left plot shows the entropy profile $S_i$ over time, and the right plot shows the entanglement profile $E_N(\rho(t))$. 
    }
    \label{fig:quantum_dynamics}
\end{figure}

\begin{remark}
Each covariance matrix $\widehat{\M{M}}^{(i)}=\M{M}(t_i)$ is real, but the inverse DFT of the family $\{\widehat{\M{M}}^{(i)}\}_{i=1}^{p}$ need not be real in the spatial domain. 
The example is therefore formulated as a Fourier-domain encoding of a sampled covariance family rather than as a tensor-valued physical trajectory. 
This is compatible with the present T-product framework over $\C$, where the quantum mechanical formulas remain the standard slice-wise Gaussian-state formulas and the tensor contribution is the T-Williamson decomposition of the encoded family.
\end{remark}

\section{Conclusion and Future Work}
\label{sec:conclusion}

We introduced Hamiltonian and symplectic tensor classes in the T-product algebra and established their basic Fourier-domain characterizations. 
We proved a constructive T-Williamson normal form for tensors whose Fourier-domain slices are real symmetric positive-definite matrices, together with an obstruction result showing that the same statement does not extend directly to arbitrary Hermitian positive-definite Fourier-domain slices under the convention $\M{S}^H \M{J}\M{S} = \M{J}$.
We also derived a real-valued recovery criterion under Fourier conjugate symmetry. Numerical experiments supported the constructive aspects of the theory and included a quantum example based on Fourier-domain tensor slices of a sampled covariance family.

Several questions remain open. Proposition~\ref{prop:complex_obstruction} shows that the Hermitian complex case cannot be treated by a direct extension of the real case. 
This suggests considering other symplectic conventions, such as the transpose-based complex symplectic group, or developing a normal form adapted to Hermitian positive-definite Fourier-domain slices.
The computational aspects also require further study. The present construction is carried out slice-wise and has arithmetic cost $O(pn^3)$. 
Faster methods may be possible when the Fourier-domain slices have additional structure, such as sparsity, low-rank perturbations, or block structure.
Finally, the framework may be applicable to other parameter-dependent structured matrix families, including those arising in control, wave propagation, and Gaussian quantum dynamics.

\section*{ Declaration of competing interest}
The authors declare that they have no competing interests.

\section*{ Data availability}
No external dataset was used.

\section*{ Acknowledgments}
All authors contributed to the conception, mathematical analysis, and writing of the manuscript. All authors read and approved the final manuscript.

\section*{ Funding}

The work of T.~Kim was supported by the National Research Foundation of Korea (NRF) grant funded by the Korea government (MSIT) (No. 2022R1A5A1033624\&\,RS-2024-00342939\&\,RS-2025-25436769). 
The work of S. L\'opez-Moreno was supported by the National Research Foundation of Korea (NRF) grant funded by the Korea government (MSIT) (RS-2024-00406152).

\bibliography{references}

\end{document}